\documentclass[11pt]{amsart}

\usepackage{amsmath,amsthm}
\usepackage{amsfonts}
\usepackage{verbatim}
\usepackage{amssymb}
\usepackage{txfonts}
\usepackage{amscd}
\usepackage{bbm}
\usepackage{hyperref}
\usepackage{mathrsfs}
\usepackage[nohug]{diagrams}
\usepackage{dsfont}

\newcommand{\setd}[2]{\,\left\{#1\ \colon\ #2\right\}}

\newtheorem{theorem}{Theorem}[section]

\newtheorem{lemma}[theorem]{Lemma}
\newtheorem{proposition}[theorem]{Proposition}
\newtheorem{definition}[theorem]{Definition}
\newtheorem{question}[theorem]{Question}
\newtheorem{remark}[theorem]{Remark}

\newcommand{\supp}{\operatorname{supp}}

\newcommand{\RR}{\mathbb{R}}

\newcommand{\NN}{\mathbb{N}}

\newcommand{\C}{\mathscr{C}}
\newcommand{\X}{\mathcal{X}}

\newcommand{\EL}{\mathcal{L}}

\title[Invariant expectations and bounded cohomology]{Invariant expectations and vanishing of bounded cohomology for exact groups}
\author{Ronald G. Douglas}
\author{Piotr W. Nowak}
\address{Department of Mathematics, Texas A\&M University, College Station, TX 77843}
\email{rdouglas@math.tamu.edu, pnowak@math.tamu.edu}

\thanks{During this research the first author was partially supported by NSF grant DMS-0600865. The second author was partially supported by NSF grant DMS-0900874.}

\begin{document}
\maketitle

\begin{abstract}
We study exactness of groups and establish
a characterization of exact groups in terms of the existence of a continuous linear operator,
called an invariant expectation,
whose properties make it a weak counterpart  of an invariant mean on a group.
We apply this operator to show that exactness of a finitely generated group $G$ 
implies the vanishing of
the bounded cohomology of $G$ with coefficients in a new class of modules, 
which are defined using the Hopf algebra structure of $\ell_1(G)$. 
\end{abstract}

\newcounter{aux1}
\newcounter{aux2}

\section{Introduction}

Exactness is a weak amenability-type property of finitely generated groups.
It was defined in \cite{kirchberg-wassermann-1} in terms of properties of the 
minimal tensor product of the reduced group $C^*$-algebra.
Similar to amenability, exactness has a few equivalent definitions which are each of
separate interest in different areas of mathematics. 
In particular, exactness is equivalent to the  existence of a topologically
amenable action of the group on a compact space \cite{higson-roe} 
and to Yu's property A \cite{guentner-kaminker-exactness,ozawa}. 
For this reason exactness has many 
interesting applications in analysis, geometry and topology. 
Most notably, Yu \cite{yu-embeddings} proved that groups with property A
satisfy the Novikov conjecture.

Amenable groups are precisely
the ones which carry an invariant mean; that is, a functional on $\ell_{\infty}(G)$ which
is positive, preserves the identity and is invariant under the natural action of $G$ on $\ell_{\infty}(G)$. 
Invariant means allow for
averaging on amenable groups, which is precisely
what makes such groups so convenient to work with. In the case of exact groups there
was no parallel characterization and in this article we investigate the existence of such
a counterpart and some of its applications.

Coarse-geometric versions of classical notions or results in group theory can sometimes be
obtained by considering the problem ``with coefficients
in $\ell_{\infty}(G)$", where the precise meaning of this phrase varies with the context. 
For instance, the coarse  Baum-Connes conjecture 
for a group $G$ is the Baum-Connes conjecture with coefficients in the algebra
$\ell_{\infty}(G,\mathcal{K})$, where $\mathcal{K}$ denotes the compact operators on 
an infinite-dimensional, separable Hilbert space \cite{yu-coarse.baum.connes}, see also \cite{spakula}.
Similarly, in coarse geometry  the reduced crossed product $\ell_{\infty}(G)\rtimes_r G$ 
is an analogue of 
the reduced group $C^*$-algebra of $G$.
It is this point of view that  suggests the study of the space of operators from a given space into
the algebra $\ell_{\infty}(G)$, which, loosely speaking, plays the role of a ``dual with
coefficients in $\ell_{\infty}(G)$".

Let $G$ be a finitely generated group. 
Given a left Banach $G$-module $\X$ we consider the space $\EL(\X,\ell_{\infty}(G))$
of continuous linear operators from $\X$ to $\ell_{\infty}(G)$. This space is naturally 
a bounded Banach $G$-module by pre- and post-composing with the actions of $G$ on 
$\X$ and $\ell_{\infty}(G)$.
Additionally, we can identify $\EL(\X,\ell_{\infty}(G))$
with the space $\ell_{\infty}(G,\X^*)$, which is a dual space to $\ell_1(G,\X)$. 
Using this identification we can naturally equip $\EL(\X,\ell_{\infty}(G))$ with 
a weak-* topology.

An invariant expectation
on the group $G$ is then an operator $M$ from the $G$-module $\EL(\ell_u(G),\ell_{\infty}(G))$ into 
$\ell_{\infty}(G)$, where $\ell_u(G)$ is a certain Banach algebra associated to the group, obtained
as a completion of the algebraic crossed product $\ell_{\infty}(G)\rtimes_{alg} G$
(the precise definition is given in \ref{definition : the uniform convolution algebra}).
The properties of the invariant expectation are similar to properties of invariant means on groups;
namely, it is positive, unital in an appropriate sense and  a limit of elements
of a special form. The most important property of a mean, invariance, is replaced here by equivariance 
with respect to the $G$-actions.
We refer to Definition \ref{definition : invariant expectation} for
details.
\begin{theorem}\label{theorem : main 2} 
A finitely generated group $G$ is exact if and only if there exists an
invariant expectation on $G$.
\end{theorem}
The existence of an invariant expectation and its relation to exactness of groups 
relies on the properties of a certain subspace $\mathcal{W}$ of $\EL(\EL(\ell_u(G),\ell_{\infty}(G)),\ell_{\infty}(G))$.
It is the choice of this subspace that allows us to carry out approximation arguments 
in a setting suitable
for exactness.

We apply the above characterization to compute the bounded cohomology of exact groups
with coefficients in a new class of Banach modules.
The motivation for studying bounded cohomology in the context of  
exactness comes from a question of N.~Higson, who asked if there is a cohomological
characterization of exactness. The first such characterization was proved in \cite{brodzki-niblo-wright},
but in terms of a cohomology theory introduced in that article. 
One natural direction to investigate is whether the characterization of amenability in terms
of bounded cohomology, due to B.E.~Johnson \cite{johnson-memoir}, could 
be generalized to our setting. 
Here we show, using invariant expectations,
that bounded cohomology with coefficients  in a class of modules defined below,
vanishes for exact groups. Characterizations of exactness via vanishing of bounded cohomology
were proved later in \cite{bnnw} and \cite{monod-amenable}, independently, using some of the ideas
presented here.

The space $\EL(\X,\ell_{\infty}(G))$, in addition to being a $G$-module, 
also carries a natural structure of an $\ell_{\infty}(G)$-module, given
by multiplying the image of an operator $T$ by an element of $\ell_{\infty}(G)$. It is
the existence of this second structure that is essential in our considerations.
A $G$-submodule $\mathcal{E}\subseteq \EL(\X,\ell_{\infty}(G))$, which additionally 
is an $\ell_{\infty}(G)$-module in this sense, will be called a Hopf $G$-module, 
as its two module structures are  defined by the actions of a Hopf-von Neumann algebra 
$\ell_{\infty}(G)$
and its predual Banach algebra $\ell_1(G)$.
By $H_b^1(G,\mathcal{E})$ we denote the bounded cohomology group of $G$ 
in degree 1 with coefficients in a Banach $G$-module $\mathcal{E}$.

\begin{theorem}\label{theorem : vanishing}
Let $G$ be a finitely generated group. If $G$ is exact then
$H^1_b(G,\mathcal{E})=0$ for any weak-* 
closed Hopf $G$-module.
\end{theorem}

Bounded cohomology allows the use of dimension shifting techniques: given
a $G$-module $M$ there is a ``shifting module" $\Sigma M$ such that 
$H^{n+1}_b(G,M)= H^n_b(G,\Sigma M)$. An argument due to 
N.~Monod shows that shifting modules of weak-* closed 
Hopf $G$-modules are again weak-* closed Hopf $G$-modules, which allows
to conclude that all the higher bounded cohomology groups with coefficients in
dual Hopf $G$-modules
also vanish for exact groups.\\

The above results, after being circulated in preprint form, 
were followed by a  rapid development of the relation between bounded cohomology and exactness.
The papers \cite{bnnw} and \cite{monod-amenable} mentioned earlier, as well as \cite{bnnw-homology} and \cite{douglas-nowak}, contain related results.

We would like to thank Nicolas Monod for many 
valuable comments, as well as suggesting
the name Hopf modules and for allowing us to include Proposition 
\ref{proposition : vanishing for higher groups} in this paper. We also thank the referee
for helpful comments.

\tableofcontents

\section{Algebras, duality and topologies}

Let $G$ be a discrete group generated by a finite set $S$, which is symmetric; that is, $S=S^{-1}$.
All Banach spaces we discuss are over $\RR$.
Since in most of our arguments we will view the algebra $\ell_{\infty}(G)$ as an algebra of 
coefficients, throughout the article we will shorten the notation to $\C=\ell_{\infty}(G)$.
We denote by $1_G$ the identity
in $\ell_{\infty}(G)$ and by $\Vert \cdot \Vert_{\C}$ the usual supremum norm.
The algebra $\C$ is equipped with a natural $G$-action,
\begin{equation}\label{equation : action on ell_infty}
(g*f)(h)=f(g^{-1}h)
\end{equation}
for $g,h\in G$ and $f\in {\C}$. 
For a function $f:G\times G\to \RR$ we will denote
$$f_g=f(g,\cdot),$$
so that $f_g$ is a real function on $G$. The set of those $g\in G$ 
for which $f_g\neq 0$ is called the support of $f$ and denoted $\supp f$.

\subsection{The uniform convolution algebra $\ell_u(G)$}
Let 
$${C_c(G,\C)}=\setd{f:G\to \C}{\#\supp f<\infty},$$
where $\supp f$ denotes the support of $f$.
By the above, ${C_c(G,\C)}$ can be viewed as a linear subspace of the space 
of bounded functions
on $G\times G$, each of which vanishes outside of $K\times G$, for some finite $K\subseteq G$.
${C_c(G,\C)}$ is a linear space and we equip it with the norm 
$$\Vert f\Vert_u=\left\Vert \sum_{g\in G} \vert  f_g\vert\right\Vert_{\C}.$$
The space ${C_c(G,\C)}$ is, in a natural way, a subspace of  
$\left( \bigoplus_{g\in G}\ell_1(G)\right)_{\infty}$, the infinite direct sum of
copies of $\ell_1(G)$
 with the norm 
$\Vert \eta \Vert=\sup_{g\in G}\Vert \eta(g)\Vert_1$, where $\eta:G\to \ell_1(G)$, 
and these norms agree on ${C_c(G,\C)}$.
We define multiplication on ${C_c(G,\C)}$ by the formula
$$(f\star f')_g=\sum_{h\in G}f_h(h*f'_{h^{-1}g})$$
and an involution $f^*_g=g*f_{g^{-1}}$, which turns $C_c(G,\C)$ into the algebraic crossed product
$\C\rtimes_{alg} G$.
It can be easily seen that
$\Vert f\star f'\Vert_u\le\Vert f\Vert_u\Vert f'\Vert_u$ for $f,f'\in {C_c(G,\C)}$.

\begin{definition}\label{definition : the uniform convolution algebra}
The uniform convolution algebra, denoted $\ell_u(G)$, is the completion of ${C_c(G,\C)}$
with respect to the norm $\Vert \cdot\Vert_u$.
\end{definition}

Note that there is a natural isometric inclusion of $\ell_1(G)$ in $\ell_u(G)$. 
Indeed, consider $f\in \ell_1(G)$ and define $\xi_g=f(g)1_G$.

For each element $g\in G$ consider  $\delta_g\in \ell_u(G)$ given by
$$(\delta_g)_h=\left\lbrace
\begin{array}{ll}
1_G&\text{ if } h=g,\\
0&\text{ otherwise.}
\end{array}
 \right.$$
We have $\delta_{gh}=\delta_g\star \delta_h$ and there is a natural action of 
$G$ on $\ell_u(G)$ by isometries, also denoted by $\star$, such that
\begin{equation}\label{equation : action star on ell_u}
g\star \xi=\delta_g\star \xi,
\end{equation}
for $\xi\in \ell_u(G)$.
Observe also that $\delta_e$, where $e\in G$ is the identity element, is the unit in $\ell_u(G)$.

\subsection{$G$-duality for $\ell_u(G)$ and $\ell_{\infty}(G,\C)$}

Consider the Banach space 
$$\ell_{\infty}(G,\C)=\setd{f:G\to\C}{\sup_{g\in G}\Vert f_g\Vert_{\C}<\infty}.$$
We denote by $\mathds{1}_G$ the identity 
in $\ell_{\infty}(G,\C)$: $(\mathds{1}_G)_h=1_G$ for every $h\in G$.

This space is naturally isometrically isomorphic to $\ell_{\infty}(G\times G)$;
however, the above notation has advantages in our setting.
The space $\ell_{\infty}(G,\C)$ is a left $G$-module with an action given by 
\begin{equation}\label{equation : action by double convolution}
(g\odot f)_h=g*f_{g^{-1}h},
\end{equation}
where $g,h\in G$.  There is a natural inclusion 
$\ell_u(G)\subseteq \ell_{\infty}(G,\C)$ and the two actions $\star$ and $\odot$ 
agree on $\ell_u(G)$.

There exists a natural $\C$-valued pairing between the elements of $\ell_u(G)$ and 
$\ell_{\infty}(G,\C)$,
$\langle \cdot,\cdot\rangle_{\C}:\ell_u(G)\times \ell_{\infty}(G,\C)\to \C$ given by 
\begin{equation}\label{equation : pairing}
\langle \xi,f\rangle_{\C}=\sum_{g\in G}\xi_g f_g.
\end{equation}
It is well-defined since 
$$\vert \langle \xi,f\rangle_{\C}\vert\le \sum_{g\in G}\vert \xi_g\vert\, \vert f_g\vert\le \Vert f\Vert_{\ell_{\infty}(G,\C)}\Vert \xi\Vert_{u}1_G.$$

\begin{lemma}\label{lemma: conjugation by g in the duality}
For $\xi\in\ell_u(G)$ and $f\in \ell_{\infty}(G,\C)$ we have
$$\langle g\star\xi,f\rangle_{\C}=g*\langle \xi,g^{-1}\odot f\rangle_{\C}.$$
\end{lemma}
\begin{proof}
We have 
\begin{eqnarray*}
\langle g\star \xi, f\rangle_{\C}(h)&=&\sum_{k\in G} \left(g\star\xi_k(h)\right)f_k(h)\\
&=&\sum_{k\in G}\xi_{g^{-1}k}(g^{-1}h)f_k(h).
\end{eqnarray*}
On the other hand,
\begin{eqnarray*}
\left(g*\langle \xi, g^{-1}\odot f\rangle_{\C}\right)(h)&=&\langle \xi, g^{-1}\odot f\rangle_{\C}(g^{-1}h)\\
&=&\sum_{k\in G}\xi_k(g^{-1}h)\,\left((g^{-1}\odot f)_{k}(g^{-1}h)\right)\\
&=&\sum_{k\in G}\xi_k(g^{-1}h)\,\left(f_{gk}(g g^{-1}h)\right)
\end{eqnarray*}
Substituting $gk=k'$ we see that the two expressions are equal.
\end{proof}

\subsection{Weak topologies}

Let $\X$ be a Banach space.  One of the main objects of our study will be the module
$\EL(\X,\C)$ of bounded linear maps from $\X$ to $\C$, with its natural
operator norm, which we denote by $\Vert\cdot\Vert_{\EL}$.

If not stated otherwise we consider $\C$ as a dual of $\ell_1(G)$ with its natural weak-* topology.
We will denote weak-* limits in $\C$ by $w^*-\lim$.
The action $*$ defined in (\ref{equation : action star on ell_u}) of $G$ on $\C$ is weak-*-continuous.

\subsubsection*{The weak-* topology on $\EL(\X,\C)$}

The space $\EL(\X,\C)$ is isometrically isomorphic to the space $\ell_{\infty}(G,\X^*)$,
where the isomorphism $I:\EL(\X,\C)\to \ell_{\infty}(G,\X^*)$ is given by 
$$(I(T)_g)(x)=(T(x))_g$$
for  $T\in \EL(\X,\C)$, $x\in X$ and $g\in G$. 
Thus our space $\EL(\X,\C)$ can be identified as
the Banach space dual of $\ell_1(G,\X)$ and as such it can be naturally equipped with 
a weak-* topology. 

The following  description of the weak-* topology on $\EL(\X,\C)$ is convenient in our 
setting.

\begin{proposition}\label{corollary : weak-C topology is pointwise weak* topology}
Let $\X$ be a Banach space and let $\{T_{\beta}\}$ be a net in $\EL(\X,\C)$.
The following conditions are equivalent:
\begin{enumerate}
\renewcommand{\labelenumi}{\normalfont{(\alph{enumi})}}
\item ${\C}-\lim_{\beta} T_{\beta}=T$,
\item $w^*-\lim_{\beta} T_{\beta}(x)=T(x)$ in $\C$ for every $x\in \X$.
\end{enumerate}
\end{proposition}

\subsubsection*{The weak topology on $\X$}
Every element $\xi\in \X$ defines a map $\hat{\xi}:\EL(\X,\C)\to \C$ by the formula
$$\hat{\xi}(T)=T(\xi)$$
for every $T\in \EL(\X,\C)$. This defines a natural embedding
$$i:\X \to \EL\left(\EL(\X,\C),\C \right).$$ We denote the natural norm on 
$\EL\left(\EL(\X,\C),\C \right)$ by $\Vert \cdot \Vert_{\EL\EL}$. Since the dual space
$\X^*$ of $\X$ is naturally embedded in $\EL(\X,\C)$ by defining
$T(x)=\varphi(x)1_G$, where $\varphi\in \X^*$, we easily see that
the embedding is isometric.

\begin{definition}\label{definition : weak topology}
Let $\X$ be a Banach space. The weak topology on $\X$ is the restriction  to $\X$ of the
weak-* topology on $\EL(\EL(\X,\C),\C)$.
\end{definition}
This gives the following descrition
\begin{proposition}\label{corollary : description of weak convergence}
A net $\{x_{\beta}\}$ of elements of $\X$ converges 
 weakly to $x\in \X$ if and only if for every $T\in \EL(\X,\C)$ we have
$T(x)=w^*-\lim_{\beta} T(x_{\beta})$.
\end{proposition}
The weak topology of Definition \ref{definition : weak topology}
is the same as the weak topology on $\X$ in the classical sense. 
For our purposes it suffices to see that  the weak topology in the sense of 
Definition \ref{definition : weak topology}
is formally stronger than the weak topology on $\X$ in the classical sense.

\section{Exactness and invariant expectations}
\subsection{Exact groups}
The term \emph{exact group} originates in the theory of $C^*$-algebras. 
However, in the last decade many new characterizations 
were discovered and our use of this term is not restricted strictly 
to the $C^*$-algebraic definition.

Originally exact groups were defined by Kirchberg and Wassermann, see \cite{kirchberg-wassermann-1,kirchberg-wassermann-2} in the study of group $C^*$-algebras.
A $C^*$-algebra $\mathcal{A}$ is exact if given any exact sequence
$$0\longrightarrow \mathcal{I}\longrightarrow \mathcal{B}\longrightarrow \mathcal{B}/\mathcal{I}\longrightarrow 0$$
the sequence 
$$0\longrightarrow \mathcal{I}\otimes_{\min}\mathcal{A} \longrightarrow \mathcal{B}\otimes_{\min}\mathcal{A}\longrightarrow \mathcal{B}/\mathcal{I}\otimes_{\min}\mathcal{A}\longrightarrow 0$$
remains exact. Note that the maximal tensor product always preserves short exact sequences
in the above sense.
Exactness of a $C^*$-algebra is weaker than its nuclearity. Indeed, $\mathcal{A}$ is nuclear if 
$\mathcal{B}\otimes_{\min}\mathcal{A}=\mathcal{B}\otimes_{\max}\mathcal{A}$ for any $C^*$-algebra $\mathcal{B}$.

A group $G$ is called exact if its reduced $C^*$-algebra $C^*_r(G)$ is exact in the above 
sense.
We refer to \cite{brown-ozawa,kirchberg-wassermann-1,kirchberg-wassermann-2,willett} for details.

Exactness of a group turned out to be equivalent to property A of Yu due
to work of Guentner and Kaminker  \cite{guentner-kaminker-exactness}
and, subsequently, Ozawa \cite{ozawa}. Property A was introduced in \cite{yu-embeddings}
as a condition sufficient to enable one to embed a group (or, more generally, a metric space) 
coarsely into a Hilbert space. At present there are no known examples of groups which 
embed coarsely into the Hilbert space but do not have property A (see, however, \cite{nowak-nonpropa}).
Property A in \cite{yu-embeddings} was defined in terms of a F\o lner-type condition
which highlights the fact that it can be viewed as a weak amenability-type property.
We refer to \cite{nowak-yu,roe-lectures,willett} for an introduction to property A.

In \cite{higson-roe} Higson and Roe characterized property A and exactness of a group $G$
in terms of topologically
amenable actions on the Stone-\v{C}ech compactification of $G$.
We will use a  version of the characterization from \cite{higson-roe} 
as our definition of exactness. 
\begin{definition}\label{definition : exactness}
A finitely generated  group $G$ is exact if for every $\varepsilon>0$ there exists an 
element $\xi\in \ell_u(G)$ such that 
\begin{enumerate}
\renewcommand{\labelenumi}{\normalfont{(\alph{enumi})}}
\item $\xi$ is finitely supported; that is, $\xi_g=0$ for all but finitely many $g\in G$,
\item $\xi$ is an $\C$-valued probability measure; that is, $\xi_g\ge0$ for all $g\in G$ and
$\sum_{g\in G}\xi_g=1_G$, and
\item $\xi$ is $\varepsilon$-invariant; that is,  $\Vert \xi-s\star\xi\Vert_u\le \varepsilon$ for every generator 
$s\in S$.
\end{enumerate}
\end{definition}

Exactness has numerous consequences in the theory of $C^*$-algebras, index theory and 
geometric group theory.
In particular, Yu proved that if $G$ is has property A or, equivalently, is exact, 
then the coarse Baum-Connes conjecture holds for $G$ \cite{yu-embeddings}.
This on the other hand implies the Novikov conjecture for $G$, the zero-in-the-spectrum
 conjecture and has applications to the positive scalar curvature problem.
More recently exactness was related to isoperimetric inequalities on
finitely generated groups and quantitative invariants like decay of the heat kernel 
and type of asymptotic dimension \cite{nowak-profiles,nowak-advances}. 

Exact groups constitute a very large class of groups. Most notably it includes all amenable groups, 
hyperbolic groups (both in \cite{yu-embeddings}), linear groups \cite{guentner-higson-weinberger}.
We refer to \cite{willett} for a more complete list. 
The task of constructing a group which is not exact turns out to be a
difficult one. At present only one family of examples is known, Gromov's random groups 
\cite{gromov-random}.
The question how to find new examples of groups which would not be
exact is still open.

\subsection{Invariant expectations in $\EL(\EL(\ell_u(G),\C),\C)$ }
A Banach space $\X$ is said to be a bounded left Banach $G$-module if there is 
a homomorphism $\Phi:G\to \EL(\X,\X)$
such that $$\sup_{g\in G}\Vert \Phi(g)\Vert_{\EL(\X,\X)}<\infty.$$
If $\X$ is a left $G$-module, then we will denote the action of $g\in G$ by $gx$ for $x\in \X$.

If $\X$ is a left  $G$-module then $\EL(\X,\C)$ is a bounded left $G$-module
with the left  action of $G$ given by pre- and post-composing with the actions
of $G$:
\begin{eqnarray}\label{equation : left G action on L(E,A)}
(g\cdot T)(x)&=&g*\left(T(g^{-1}x)\right),
\end{eqnarray}
for $T\in \EL(\X,\C)$, $x\in \X$ and $g\in G$.

\begin{lemma}\label{lemma : the action cdot is weak-C continuous}
The above action is weak-* continuous.
\end{lemma}
\begin{proof}
If $T={\C}-\lim_{\beta}\, T_{\beta}$, then 
\begin{eqnarray*}
w^*-\lim_{\beta} g\cdot T_{\beta}(x)& =&w^*-\lim_{\beta} g*(T_{\beta}(g^{-1}x))\\
&=&g*(w^*-\lim_{\beta}  T_{\beta}(g^{-1}x))\\
&=&g\cdot T(x),
\end{eqnarray*}
where the second equality follows from weak-* continuity of the action on $\C$.
\end{proof}

Observe that given $f\in \ell_{\infty}(G,\C)$ the
pairing $\langle \xi,f \rangle_{\C}$ for $\xi \in \ell_u(G)$ 
gives naturally an operator in $\EL(\ell_u(G),\C)$. Moreover,

\begin{lemma}
$\Vert f\Vert_{\ell_{\infty}(G,\C)}=\Vert f\Vert_{\EL}$. In particular,
the space $\ell_{\infty}(G,\C)$ is isometrically embedded in 
$\EL(\ell_u(G),\C)$.
\end{lemma}
\begin{proof}
The estimate $\Vert  f\Vert_{\EL}\le \Vert f\Vert_{\ell_{\infty}(G,\C)}$ follows easily. 
To see the converse observe that for every $\varepsilon>0$ there is $g\in G$ such that 
$\Vert f\Vert_{\ell_{\infty}(G,\C)}\le \Vert f_g\Vert_{\C} +\varepsilon$. 
Then $\langle \delta_g,f\rangle_{\C}=f_g$ and the required 
inequality follows by taking $\varepsilon$ converging to 0.
\end{proof}

The following lemma shows that $\ell_{\infty}(G,\C)$ is also a $G$-submodule of $\EL(\ell_u(G),\C)$.
\begin{lemma}\label{lemma : actions cdot and star agree o ell_infty}
The actions $\cdot$ and $\odot$ agree on $\ell_{\infty}(G,\C)\subseteq \EL(\ell_u(G),\C)$.
\end{lemma}
\begin{proof}
By Lemma \ref{lemma: conjugation by g in the duality}, for any $f\in \ell_{\infty}(G,\C)$ and 
$\xi\in \ell_u(G)$ we have 
$$(g\cdot f)(\xi)=g*(\langle g^{-1}\star \xi,f\rangle_{\C})=(g\odot f)(\xi).$$
Taking $\xi=\delta_h$ for any $h$ gives the equality. 
\end{proof}

The action of $G$ on the $G$-module $\EL(\EL(\ell_u(G),\C),\C)$ will now be denoted by
$\bullet$ to distinguish it from the action $\cdot$ on $\EL(\ell_u(G),\C)$:
\begin{eqnarray}\label{equation : odot action on ELEL}
(g\bullet \Xi)(T)&=&g*\left(\Xi(g^{-1}\cdot T) \right),
\end{eqnarray}
for $\Xi\in \EL(\EL(\ell_{u}(G),\C),\C)$ and $T\in \EL(\ell_u(G),\C)$.

\begin{lemma}\label{lemma : actions star and odot agree} 
The actions $\bullet$ and $\star$ agree on $\ell_u(G)\subseteq \EL(\EL(\ell_u(G),\C),\C)$.
\end{lemma}
\begin{proof}
We have
\begin{eqnarray*}
(\widehat{g\star\xi})(T)&=&T(g\star\xi)\\
&=&g*\left(g^{-1}*(T(g\star\xi))\right)\\
&=&g*\left(g^{-1}\cdot T(\xi)\right)\\
&=&g*\left(\hat{\xi}(g^{-1}\cdot T)\right)\\
&=&(g\bullet \hat{\xi})(T),
\end{eqnarray*}
for every $T\in \EL(\ell_u(G),\C)$, $\xi\in\ell_u(G)$.\end{proof}

Consider the space
$$\mathcal{W}_{00}=\setd{\xi\in \ell_u(G)}{\#\supp\xi<\infty\text{ and }\langle \xi,\mathds{1}_G\rangle_{\C}=c1_G \text{ for some } c\in \RR}$$
and let $\mathcal{W}_0$ be the closure of $\mathcal{W}_{00}$ in the norm topology in $\ell_u(G)$.

\begin{definition}
Define the subspace $\mathcal{W}\subseteq \EL(\EL(\ell_u(G),\C),\C)$ to be 
the weak-* closure of $\mathcal{W}_{00}$. 
\end{definition}
Clearly, $\mathcal{W}$ is a Banach subspace of $\EL(\EL(\ell_u(G),\C),\C)$. Moreover,
it has a natural structure of a $G$-module.

The above setup allows us to prove now the main theorem characterizing exactness. Amenable
groups are known to be characterized by a F\o lner and Reiter conditions, which correspond
to our Definition \ref{definition : exactness}  of exactness (see \cite{paterson,pier-groups}).
Another standard definition of amenability is through
the existence of an invariant mean on the group. 
The next definition provides a weak version of the invariant mean.

\begin{definition}\label{definition : invariant expectation}
Let $G$ be a finitely generated group. 
An invariant expectation on $G$ is a bounded linear operator 
$M:\EL(\ell_u(G),\C)\to\C$ which satisfies 
\begin{enumerate}
\renewcommand{\labelenumi}{\normalfont{(\alph{enumi})}}
\item $M\in \mathcal{W}$, 
\item $M(\mathds{1}_G)=1_G$, and
\item $M$ is $G$-invariant; that is,  $g\bullet  M = M$ for every $g\in G$.
\end{enumerate}
\end{definition}

The importance of the notion of an invariant expectation is in its
relation to exactness of groups described in Theorem \ref{theorem : main 2}, whose
statement we recall from the introduction.
\setcounter{aux1}{\value{section}}
\setcounter{aux2}{\value{theorem}}
\setcounter{section}{1}
\setcounter{theorem}{0}
\begin{theorem}\label{theorem : exactness implies invariant operator}
A finitely generated group $G$ is exact if and only if 
there exists an invariant expectation on $G$.
\end{theorem}

\setcounter{section}{\value{aux1}}
\setcounter{theorem}{\value{aux2}}

\begin{proof}
Consider a sequence $\{\xi_n\}$ where 
$\xi_n$ is obtained from the definition of exactness with $\varepsilon=\frac{1}{n}$. Each $\xi_n$ 
is an element of $\ell_u(G)$ and as such induces a continuous linear map
$\hat{\xi}_n\in \EL(\EL(\ell_u(G),\C),\C)$. We consider this last space with the 
weak-* topology described in the previous section. 
Since, by the Banach-Alaoglu theorem, the 
unit ball of the space $\EL(\EL(\ell_u(G),\C),\C)$
is compact with this topology,  the sequence $\{\hat{\xi}_n\}$ has a convergent 
subnet $\{\hat{\xi}_{\beta}\}$ and we define 
$$ M ={\C}-\lim_{\beta} \hat{\xi}_{\beta},$$
which is equivalent to
$$M(T)=w^*-\lim_{\beta}\hat{\xi}_{\beta}(T)$$
for every $T\in \EL(\ell_u(G),\C)$, by Corollary \ref{corollary : weak-C topology is pointwise weak* topology}. We will show that $M$ is an invariant expectation on $G$.

Clearly, $M\in \mathcal{W}$ and, in particular, since $\langle \xi_{\beta},\mathds{1}_G\rangle_{\C}=1_G$ 
for every $\beta$, it follows that $ M (\mathds{1}_G)=1_G$. 

By lemmas \ref{lemma : the action cdot is weak-C continuous} and \ref{lemma : actions star and odot agree} we have for $T\in \EL(\ell_u(G),\C)$ and any generator $s\in S$,
\begin{eqnarray*}
(s\bullet  M) (T)&=&s*\left(M(s^{-1}\cdot T)\right)\\
&=&s*\left(w^*-\lim_{\beta} \left(\hat{\xi}_{\beta}(s^{-1}\cdot T) \right)\right)\\
&=&w^*-\lim_{\beta} s*\left(\hat{\xi}_{\beta}(s^{-1}\cdot T) \right)\\
&=&w^*-\lim_{\beta} (s\bullet\hat{\xi}_{\beta})(T)\\
&=&w^*-\lim_{\beta} (\widehat{s\star\xi})(T).
\end{eqnarray*}
Thus 
\begin{equation}\label{equation : G invariance of the mean}
 \left(M -s\bullet  M\right)(T)=w^*-\lim_{\beta}\ (\hat{\xi}_{\beta}-\widehat{s\star\xi_{\beta}})(T)
\end{equation}
and  we have
\begin{eqnarray*}
\Vert (\hat{\xi}_{\beta}-\widehat{s\star\xi_{\beta}})(T)\Vert_{\C}
&\le&\Vert \hat{\xi}_{\beta}-\widehat{s\star\xi}_{\beta}\Vert_{\EL\EL}\Vert T\Vert_{\mathcal{L}}\\
&\le&\Vert \hat{\xi}_{\beta}-\widehat{s\star\xi}_{\beta}\Vert_u\Vert T\Vert_{\mathcal{L}}\\
&\le&\varepsilon_{\beta}\Vert T\Vert_{\mathcal{L}}.
\end{eqnarray*}
for every $n$. Since $\varepsilon_{\beta}$ tends to 0 this implies that the weak-* limit
 in (\ref{equation : G invariance of the mean}) also is 0 for every $T$. This proves $G$-invariance
 of $M$.\\

Conversely, let $M$ be an invariant expectation on $G$. Since $M$ is in $\mathcal{W}$ we can approximate it in the weak-* topology on the module
$\EL\left(\EL(\ell_u(G),\C),\C)\right)$ by finitely supported elements of $\mathcal{W}_{00}$.
More precisely, there exists a net $\{\xi_{\beta}\}$ such that $\xi_{\beta}\in \mathcal{W}_{00}$ 
 and 
$$w^*-\lim_{\beta}\ (\hat{\xi}_{\beta}-s\bullet\hat{\xi}_{\beta})(T)=0\ \ \text{ in } \C$$
for every $T\in \EL(\ell_u(G),\C)$ and $s\in S$.
Since the actions $\bullet$ and $\star$ agree on $\ell_u(G)$, this is the same as
\begin{equation}\label{equation : weak convergence is the same as L-convergence}
w^*-\lim_{\beta} T(\xi_{\beta}-s\star\xi_{\beta})=0
\end{equation}
for every $T\in \EL(\ell_u(G),\C)$ and $s\in S$.
Moreover, $\xi_{\beta}$ satisfy
$$\langle \xi_{\beta},\mathds{1}_G\rangle_{\C}=c_{\beta}1_G,$$
where the net of real numbers $\{c_{\beta}\}$ converges to 1.
By passing to a cofinal subnet, which we will also denote by $\xi_{\beta}$, we can assume that
$$\langle \xi_{\beta},\mathds{1}_G\rangle_{\C}\ge\dfrac{1}{2}1_G.$$

We will now ensure condition (c) of Definition \ref{definition : exactness} and construct a 
sequence $\xi'_{n}$ of finitely supported elements 
in $\mathcal{W}_{00}$ with similar properties to those of $\xi_{\beta}$ and such that,
additionally,
$\Vert \xi_n-s\star\xi_n\Vert_u$ tends to $0$ uniformly for all $s\in S$. 
Consider the space $\bigoplus_{s\in S}\ell_u(G)$ with the norm 
$$\Vert \sigma\Vert_{\oplus u}=\sup_{s\in S}\Vert \sigma_s\Vert_u$$
where $\sigma\in \bigoplus_{s\in S}\ell_u(G)$, $\sigma=\oplus_{s\in S} \sigma_s$.

For each $\beta$ consider the direct sum 
$$\sigma_{\beta}=\oplus_{s\in S}\left(\xi_{\beta}-s\star\xi_{\beta}\right).$$ 
From equation (\ref{equation : weak convergence is the same as L-convergence})
we deduce that for each $s\in S$, the net $\{(\sigma_{\beta})_s\}$ converges in the
weak topology on 
$\ell_u(G)$. Namely, for
each generator $s\in S$, we have
$$\varphi(\xi_{\beta}-s\star\xi_{\beta})\longrightarrow 0$$
for every linear functional $\varphi\in \ell_u(G)^*$. Since the dual spaces 
satisfy the equality $\left(\bigoplus_{s\in S}\ell_u(G)\right)^*=\bigoplus_{s\in S}\ell_u(G)^*$,
the net
$\sigma_{\beta}$ converges in the weak topology on $\bigoplus_{s\in S}\ell_u(G)$. 
Now Mazur's lemma applied to 
the closed convex hull $\Delta$ of  the $\{\sigma_{\beta}\}$ gives that the weak and strong 
closures of $\Delta$ are the same and, in particular, $0\in \Delta$. 
Thus we can approximate $0$ by finite
convex combinations of $\sigma_{\beta}$ in the norm topology on $\bigoplus_{s\in S}\ell_u(G)$.
This means that there
exists a sequence $\{\sigma_n'\}$ such that for each $n\in\NN$ the element $\sigma_n'$ is a finite
convex combination of the $\{\sigma_{\beta}\}$ and with the property that 
$\sigma_n'$ converges strongly to 0 in $\bigoplus_{s\in S}\ell_u(G)$. 
There is a corresponding sequence $\xi_n'\in\ell_u(G)$ such that 
$\sigma'_n=\oplus_{s\in S}\left(\xi_n'-s\star\xi_n'\right)$ which satisfies
$$\sup_{s\in S}\Vert \xi_{\beta}-s\star\xi_{\beta}\Vert_u\longrightarrow 0.$$
Since each $\xi_n'$ is a  finite convex combination of the $\{\xi_{\beta}\}$, we have
\begin{eqnarray*}
\langle \xi_n', \mathds{1}_G\rangle_{\C}&=&\left\langle\sum_{i=1}^{k} c_i \xi_{{\beta}_i}, \mathds{1}_G\right\rangle_{\C}\\
&=&\sum_{i=1}^{k} c_i \left\langle\xi_{{\beta}_i}, \mathds{1}_G\right\rangle_{\C}\\
&\ge&\sum_{i=1}^{k}c_i\left(\dfrac{1}{2}1_G\right)\\
&\ge&\dfrac{1}{2}1_G,
\end{eqnarray*}
where $c_i\ge 0$ and $\sum c_i=1$.
The elements $\xi_n'$ are also finitely supported and belong to $\mathcal{W}_{00}$. Thus the sequence 
$\xi_n'$ satisfies conditions (a) and (c) of Definition \ref{definition : exactness}.

We need to ensure condition (b) from Definition \ref{definition : exactness}. 
To this end consider the sequence $\{\zeta_n\}$ defined as
$$(\zeta_n)_g=\dfrac{\vert(\xi'_n)_g\vert}{\sum_{h\in G} \vert (\xi'_n)_h\vert}.$$
Then $\zeta_n\in \mathcal{W}_{00}$ and we have
$$\langle \zeta_n,\mathds{1}_G\rangle_{\C}= 1_G.$$
Since 
$$\sum_{g\in G}\vert(\xi'_n)_g\vert \ge \Big\vert \sum_{g\in G}(\xi'_n)_g\Big\vert\ge \dfrac{1}{2}1_G$$
we conclude that
\begin{equation}\label{equation : the numerator by which we normalize the xis}
\left\Vert \dfrac{1}{ h*\left(\sum_{g\in G} \vert (\xi'_n)_g\vert\right)}\right\Vert_{\C}\le 2
\end{equation}
for any $ h\in G$.

It remains to show that $\zeta_n$ satisfy the conditions of Definition  \ref{definition : exactness}
Clearly, $(\zeta_n)_g\ge 0$ and $\sum_{g\in G}(\zeta_n)_g=1_G$ for every $n\in \NN$.
It is also obvious that $\zeta_n$ is finitely supported for every $n\in \NN$.
We only need to verify the approximate invariance.
\begin{lemma}
$\Vert s\star\zeta_n-\zeta_n\Vert_u\le4{\left\Vert s\star\xi'_n-\xi'_n\right\Vert_u}$ for every generator $s\in S$.
\end{lemma}
\begin{proof}
 We have
\begin{equation}\label{equation : approximate invariance estimate}
\Vert s\star\zeta_n-\zeta_n\Vert_u=\left\Vert \dfrac{s\star\vert \xi'_n\vert}{s*\sum_{g\in G}\vert(\xi'_n)_g\vert}- \dfrac{\vert\xi'_n\vert}{\sum_{g\in G}\vert(\xi'_n)_g\vert} \right\Vert_{u}\\\\
\end{equation}
Adding a connecting term, applying the triangle inequality and (\ref{equation : the numerator by which we normalize the xis}) we obtain
\begin{eqnarray}
&\le&\left\Vert \dfrac{s\star\vert\xi'_n\vert}{s*\sum_{g\in G}\vert(\xi'_n)_g\vert} -\dfrac{\vert\xi'_n\vert}{s*\sum_{g\in G}\vert(\xi'_n)_g\vert} +\dfrac{\vert\xi'_n\vert}{s*\sum_{g\in G}\vert(\xi'_n)_g\vert} - \dfrac{\vert\xi'_n\vert}{\sum_{g\in G}\vert(\xi'_n)_g\vert} \right\Vert_{u}\\\nonumber\\
&\le &2\Vert s\star\vert \xi'_n\vert-\vert\xi'_n\vert\,\Vert_u+\left\Vert\dfrac{\vert \xi'_n\vert}{s*\sum_{g\in G}\vert(\xi'_n)_g\vert} - \dfrac{\vert\xi'_n\vert}{\sum_{g\in G}\vert(\xi'_n)_g\vert} \right\Vert_{u}.
\end{eqnarray}
The second summand, after cancellation, can be estimated as follows. 
We observe that 
\begin{eqnarray*}
\left\Vert\dfrac{\vert\xi'_n\vert}{s*\sum_{g\in G}\vert(\xi'_n)_g\vert} - \dfrac{\vert\xi'_n\vert}{\sum_{g\in G}\vert(\xi'_n)_g\vert} \right\Vert_{u}&\ \ \ \ \ \ \ \ \ \ \ \ \ \ \ \ \ \ \ \ \ \ \ \ \ \ \ \ \  &\ \ \ \ \ \ \ \ \ \ \ \
\end{eqnarray*}
\begin{eqnarray*}
\ \ \ \ \ \ \ \ \ \ \ \ \ \ \ \ \ \ \ \ \ \ \ \ \ \ \ \ \ &=&\left\Vert \left(\sum_{g\in G}\vert(\xi'_n)_g\vert\right)\dfrac{s*\sum_{h\in G}\vert(\xi'_n)_h\vert-\sum_{h\in G}\vert(\xi'_n)_h\vert}{s*\left(\sum_{g\in G}\vert(\xi'_n)_g\vert\right)\left(\sum_{g\in G}\vert(\xi'_n)_g\vert\right)}     \right\Vert_{\C}\\\\
&=&\left\Vert \dfrac{\sum_{h\in G}\vert( s\star\xi'_n)_{h}\vert-\sum_{h\in G}\vert(\xi'_n)_h\vert}{s*\left(\sum_{g\in G}\vert(\xi'_n)_g\vert\right)}  \right\Vert_{\C}\\\\
&\le&\left\Vert {\sum_{h\in G}\vert( s\star\xi'_n)_{h}\vert-\sum_{h\in G}\vert(\xi'_n)_h\vert}\right\Vert_{\C}    
\left\Vert \dfrac{1}{s*\left(\sum_{g\in G}\vert(\xi'_n)_g\vert\right)}  \right\Vert_{\C}\\\\
&\le&2{\left\Vert s\star\vert\xi'_n\vert-\vert\xi'_n\vert\right\Vert_u},
\end{eqnarray*}
where the last step follow from 
the triangle inequality and (\ref{equation : the numerator by which we normalize the xis}).
Altogether we get 
$$\Vert s\star\zeta_n-\zeta_n\Vert_u\le 4{\left\Vert s\star\vert\xi'_n\vert-\vert\xi'_n\vert\right\Vert_u}\le4{\left\Vert s\star\xi'_n-\xi'_n\right\Vert_u},$$
where the last inequality follows again from the triangle inequality.
\end{proof}
Thus the sequence $\zeta_n$ satisfies all three conditions of Definition \ref{definition : exactness}
and the group $G$ is exact.
\end{proof}

We will denote by $\mathcal{M}$ the 
subset of $\EL(\EL(\ell_{u}(G),\C),\C)$ of expectations on $G$,
meaning elements $M$ satisfying only conditions (a) and (b) of Definition 
\ref{definition : invariant expectation}, and by $\mathcal{M}^G\subseteq \mathcal{M}$ the subset of invariant
expectations.
We believe that, in general, $\mathcal{M}^G$ is an infinite set.
A natural question in this context is under what conditions the invariant expectation on $G$ is
 unique?

\begin{remark}[Automatic positivity]\normalfont\label{remark : automatic positivity}
The above proof establishes one additional property of the invariant expectation 
$M$ constructed in the ``only if" part of the proof. Namely, the 
restriction of  $M$ to $\ell_{\infty}(G,\C)$
is a positive map. Indeed, if $f\in \ell_{\infty}(G,\C)$, $f\ge 0$ then 
$\langle\xi_{\beta},f\rangle_{\C}\ge0$ for every $\beta$.
Since 
$$ M (f)=w^*-\lim_{\beta} \langle \xi_n,f\rangle_{\C}$$
and weak-* limits in $\C$ preserve positivity, we have
$ M (f)\ge 0$. However, if we start with an $M$ that is not positive in this sense, by passing through the
``if" and then  the ``only if" part of the proof we obtain a new invariant
expectation with the positivity property.
\end{remark}

We will require one more lemma about invariant expectations 
for the proof of our main results.

\begin{lemma}\label{lemma : M of a constant function is constant}
Let $G$ be an exact, finitely generated group and $M\in\mathcal{M}$ be a weak-*
limit of a net $\{\xi_{\beta}\}$ of elements satisfying conditions of 
Definition \ref{definition : exactness}.
Let $f'\in\C$ and define $f\in\ell_{\infty}(G,\C)$ by $f_g=f'$ for every $g\in G$.
Then $M(f)=f'$. 
\end{lemma}
\begin{proof}
For each $h\in G$  and every $\beta$ we have 
$$\langle \xi_{\beta},f\rangle_{\C}=\sum_{g\in G}(\xi_{\beta})_gf_g=f'\left(\sum_{g\in G}(\xi_{\beta})_g\right)=f',$$
since $\sum (\xi_{\beta})_g=1_G$ 
and this property is preserved by the weak-* limit in $\C$.
\end{proof}

\section{Bounded cohomology of exact groups}

We will now use the facts established in the previous sections to prove 
a vanishing result for bounded cohomology of exact groups. 
Recall that given a group $G$ and a bounded 
Banach $G$-module $\mathcal{E}$, a bounded cocycle is 
a map $b:G\to\mathcal{E}$ such that 
$\sup_{g\in G}\Vert b(g)\Vert<\infty$  and 
$$b(gh)=gb(h)+b(g)$$
for all $g,h\in G$. Such a cocycle 
$b$ is called a boundary if there exists an element $\phi\in \mathcal{E}$ such that
$$b(g)=g \phi-\phi$$
for every $g\in G$.
Then the bounded cohomology in degree 1 of $G$ with coefficients in $\mathcal{E}$ is the defined as
$$H^1_b(G,\mathcal{E})=Z^1(G,\mathcal{E})/B^1(G,\mathcal{E}),$$
where $Z^1(G,\mathcal{E})$ is the space of all bounded cocycles $b:G\to \mathcal{E}$ and 
$B^1(G,\mathcal{E})\subseteq Z^1(G,\mathcal{E})$ is the subspace of all boundaries 
$b:G\to \mathcal{E}$. Thus
$H^1_b(G,\mathcal{E})=0$ if and only if every bounded bounded cocycle 
$b:G\to \mathcal{E}$ is a boundary.
See \cite{dales-et-al,runde-lectures} for details in the context of Banach algebras and 
\cite{monod,monod-icm} in the context of locally compact groups. The bounded cohomology
groups of $G$ are canonically isomorphic to the Hochschild cohomology groups of 
the Banach algebra $\ell_1(G)$, with the same coefficients.

\subsection{Hopf $G$-modules} Since $\C$ is a Banach algebra, for any $\X$ the space $\EL(\X,\C)$ carries a natural structure of a $\C$-module. For $a\in\C$ and $T\in \EL(\X,\C)$ define
$$(aT)(x)=a T(x),$$
where the multiplication on the right is in $\C$. 

\begin{definition}
A  Banach $G$-module $\mathcal{E}$ is called a Hopf $G$-module 
if for some  left bounded Banach $G$-module $\mathcal{X}$ 
it is a subspace $\mathcal{E}\subseteq\EL(\X,\C)$ 
which is both a $G$-module with respect to the action of $G$
and a $\C$-module with respect to the above action of $\C$.
\end{definition}

The intuition behind the notion of Hopf $G$-modules is that they are ``large" submodules of 
$\EL(\X,\C)$, as the two structures are, loosely speaking, transverse to each other.
Indeed, consider a $G$-module $\X$. The dual space $\X^*$ has an induced $G$-module structure
and we consider 
the two natural inclusions of $\X^*$ into $\EL(\X,\C)$. The first one is obtained by mapping a functional 
$\phi$ to $\phi(x)1_G$. The second one is obtained by mapping $\phi$ to $\phi(x)1_e$, 
where $1_e$ is
the Dirac delta function  at the identity element. Note that neither of these inclusions gives rise to 
a structure of a Hopf-$G$-module on $\X^*$. Indeed, $\X^*$ is not a $\C$-submodule under the 
first one $\X^*$ and it is not a $G$-submodule under the second one.

\setcounter{aux1}{\value{section}}
\setcounter{aux2}{\value{theorem}}
\setcounter{section}{1}
\setcounter{theorem}{1}

\begin{theorem}
Let $G$ be a finitely generated group. If $G$ is exact then
 $H^1_b(G,\mathcal{E})=0$ for any weak-* 
closed Hopf $G$-module.
\end{theorem}
\setcounter{section}{\value{aux1}}
\setcounter{theorem}{\value{aux2}}

\begin{proof}
Let $\X$ be a left $G$-module, $\mathcal{E}\subseteq\EL(\X,\C)$ be as in Theorem 
\ref{theorem : vanishing}, $b:G\to \mathcal{E}$ be
a  bounded cocycle and let $C=\sup_{g\in G}\Vert b(g)\Vert$. We will show that $b$ is a boundary.
Define an operator $\Lambda:\X\to\ell_{\infty}(G,\C)$ by setting
$$\Lambda(x)_g=\left[{b}(g)\right](x),$$
for $x\in \X$.
We have 
$$\Vert b(g)(x)\Vert_{\C}\le \Vert b(g)\Vert_{\EL}\,\Vert x\Vert_{\X}\le C \Vert x\Vert_{\X},$$
so the map is well-defined and continuous.
The group $G$ is exact so,
by Theorem \ref{theorem : exactness implies invariant operator}, there exists an invariant
expectation $M$ on $G$.
Moreover,  $M$ can be chosen to be a weak-* limit of a net $\{\xi_{\beta}\}$, where
$\xi_{\beta}\in \mathcal{W}_{00}$,
as in the proof of Theorem \ref{theorem : exactness implies invariant operator}.
Since $\ell_{\infty}(G,\C)\subseteq \EL(\ell_u(G),\C)$ we define $\phi:\X\to{\C}$ by 
$$\phi(x)= M (\Lambda(x)_g).$$
The process is illustrated by the following diagram\\
$$\begin{diagram}
\X &&\pile{\rTo^{b(g)}\\ \rTo_{\phi}}&& {\C}\\
&\rdTo_{\Lambda} &  &\ruTo_{M }&\\
&& \EL(\ell_u(G),\C)&&\\\\
\end{diagram}
$$
Obviously, $\phi\in \EL(\X,\C)$ and we need to show that $\phi\in \mathcal{E}$.
By the definition of $M$, for every $x\in \X$ we have 
\begin{eqnarray*}
\phi(x)&=&w^*-\lim_{\beta}\langle \xi_{\beta},\Lambda(x)_g\rangle\\
&=&w^*-\lim_{\beta}\left(\sum_{g\in G} (\xi_{\beta})_gb(g)\right)(x).
\end{eqnarray*}
In other words,
$$\phi={\C}-\lim_{\beta} b_{\beta},$$
where $b_{\beta}=\sum_{g\in G}(\xi_{\beta})_gb(g)$.
Since $\mathcal{E}$ is a Hopf $G$-module, each $(\xi_{\beta})_g\in \C$ and only finitely many of them are non-zero,  and $b(g)\in \mathcal{E}$,
we deduce that
$b_{\beta}$ belongs to $\mathcal{E}$.
Since $\mathcal{E}$ is weak-* closed in $\EL(\X,\C)$, $\phi$ is also an element of $\mathcal{E}$.

For any $g\in G$ and $x\in \X$ we have
\begin{eqnarray*}
\left( g\cdot \phi-\phi\right)(x)&=& g*\phi( g^{-1} x)-\phi(x)\\
&=& g*\left( M (\Lambda( g^{-1} x)_{h})\right)- M (\Lambda(x)_{h}).
\end{eqnarray*}
Note that for $f\in\ell_{\infty}(G,\C)$, the invariance of $M$ can also be written in the following way
$$ g*\left( M (f_h)\right)=M((g\odot f)_h)= M \left( g*f_{ g^{-1}{h}}\right).$$
Thus
$$\left( g\cdot \phi-\phi \right)(x)= M  \left( g*\Lambda( g^{-1} x)_{ g^{-1}{h}}\right)- M (\Lambda(x)_{h})$$ and
\begin{eqnarray*}
 g*\Lambda( g^{-1}x)_{ g^{-1}{h}}&=& g*\left(\left({b}( g^{-1}{h})\right)( g^{-1}x)\right)\\
&=& g\cdot{b}( g^{-1}{h})(x)\\
&=& g\cdot\left( g^{-1}\cdot b({h})+b( g^{-1})\cdot {h}\right)(x).
\end{eqnarray*}
For cocycles we have $ g\cdot b( g^{-1})=-b( g)$.
Thus for every ${h}\in G$ we have
\begin{eqnarray*}
 g*\Lambda( g^{-1}x)_{ g^{-1}{h}}&=&\left({b}({h})-{b}( g)\right)(x)\\
&=&\Lambda(x)_{h}-{b}( g)(x).
\end{eqnarray*}
In the above expression,
$b( g)(x)$ is independent of ${h}$. Hence after applying $M$, by Lemma \ref{lemma : M of a constant function is constant}, we have $M(b( g)(x))=b( g)(x)$ and
\begin{eqnarray*}
\left( g\cdot \phi-\phi\right)(x)&=& M (\Lambda(x)_{h})-b( g)(x)- M ( 
\Lambda(x)_{h})\\
&=&-b( g)(x).
\end{eqnarray*}
Finally, setting $\Xi=-\phi$ we get $b( g)= g\cdot \Xi-\Xi$,
so that $b$ is a boundary as required.
\end{proof}

 \section{Concluding remarks}

\subsection{Dimension reduction and higher cohomology groups.}
In the case when $G$ is amenable the fact that $H^1_b(G,\mathcal{E}^*)=0$ for every 
Banach $G$-module 
$\mathcal{E}$ implies that $H_b^n(G,\mathcal{E}^*)=0$ for all $n\ge 1$.  
The method used to prove this is the dimension reduction formula in bounded cohomology,
see for example \cite{johnson-memoir},\cite[Section 10.3]{monod}, \cite[Theorem 2.4.6]{runde-lectures}.
Nicolas Monod
communicated to us the following argument showing that a similar fact is true in the case 
of exactness.
\begin{proposition}[N.~Monod]\label{proposition : vanishing for higher groups}
If $G$ is exact then for every $n\ge 1$ we have 
$H_b^n(G,\mathcal{E})=0$ for every weak-* closed
Hopf $G$-module of $\EL(\X,\C)$, where $\X$ is a left $G$-module.
\end{proposition}
\begin{proof}[Sketch of proof] 
For every module $\mathcal{E}$ we have
$$H^{n+1}_b(G,\mathcal{E})=H^n_b(G,\Sigma \mathcal{E}),$$
where $\Sigma\mathcal{E}=\ell_{\infty}(G,\mathcal{E})/\mathcal{E}$.
Note that $\Sigma\mathcal{E}$ is a dual space, namely it is the dual of  $K$, the kernel of the 
summation map $\ell_1(G, \mathcal{E}_*) \to \mathcal{E}_*$, where $\mathcal{E}_*$
denotes a given predual of $\mathcal{E}$.

There is a natural inclusion $i:\ell_{\infty}(G,\mathcal{E})\to \ell_{\infty}(G,\ell_{\infty}(\widetilde{G},\X^*))$,
where $\widetilde{G}=G$ and the notation allows one to keep track of the different copies of $G$.
We have natural isometric isomorphisms
$$\ell_{\infty}(G,\ell_{\infty}(\widetilde{G},\X^*))=\ell_{\infty}(G\times \widetilde{G},\X^*)=
\ell_{\infty}(\widetilde{G},\ell_{\infty}(G,\X^*)).$$
This descends to a canonical inclusion
$$\Sigma\mathcal{E} \subseteq \ell_\infty(\widetilde{G}, \Sigma\X^*)$$
and one can verify that, under the resulting identification, the module $\Sigma\mathcal{E}$ is a
Hopf $G$-module, with respect to  
$\ell_{\infty}(\widetilde{G})$, of $\ell_{\infty}(\widetilde{G}, \Sigma \X^*)
\simeq \EL(K,\ell_{\infty}(\widetilde{G}))$.
\end{proof}

\subsection{Relation between $\mathcal{M}$ and $\mathcal{W}$.}
It is also interesting to investigate the relation between the set of positive 
expectations $\mathcal{M}_+$ (in the sense of remark \ref{remark : automatic positivity}) and the module $\mathcal{W}$.
One possibility is that $\mathcal{W}$ is generated by $\mathcal{M}_+$ in the sense that
for every $\Xi\in \mathcal{W}$ there exist $M,M'\in \mathcal{M}_+$ and constants 
$c,c'\in [0,+\infty)$ such that
$$\Xi=cM-c'M'.$$
\begin{question}
Is $\mathcal{W}$ generated by $\mathcal{M}_+$ in the above sense?
\end{question}

\end{document}